\documentclass[a4paper, 10pt]{article}
\usepackage{amsthm}
\usepackage{amsmath,amssymb}
\usepackage{graphicx}
\usepackage{lineno}

\newtheorem{theorem}{Theorem}[section]
\theoremstyle{plain}

\newtheorem{lemma}[theorem]{Lemma}

\newtheorem{corollary}[theorem]{Corollary}

\newtheorem{question}[theorem]{Question}
 
\theoremstyle{definition}
\newtheorem{definition}[theorem]{Definition}

\begin{document}


\title{Partite Tur\'an-densities for complete $r-$uniform hypergraphs on $r+1$ vertices}

\author{Klas Markstr\"om\thanks{Department of Mathematics and Mathematical Statistics, Ume\aa\  University , SE-901 87  Ume\aa, Sweden}  and
 Carsten Thomassen\thanks{Department of Applied Mathematics and Computer Science,
Technical University of Denmark, DK-2800 Lyngby, Denmark}}

\maketitle


\begin{abstract}
	In this paper we investigate density conditions for finding a complete $r$-uniform hypergraph $K_{r+1}^{(r)}$ on $r+1$ vertices in an $(r+1)$-partite $r$-uniform 
	hypergraph $G$.  First we prove an optimal condition in terms of the densities of the $(r+1)$ induced $r$-partite subgraphs of $G$. Second,  
	we prove a version of this result where we assume that $r$-tuples of vertices in $G$ have their neighbours evenly distributed in $G$. Third, we  also prove 
	a counting result for the minimum number of copies of  $K_{r+1}^{(r)}$ when $G$ satisfies our density bound, and present some open problems.
	
	A striking difference between the graph, $r=2$, and the hypergraph, $ r \geq 3 $, cases is that in the first case both the existence threshold and the 
	counting function are non-linear in the involved densities, whereas for hypergraphs they are given by a linear function. Also, the smallest density of the $r$-partite parts needed to 
	ensure the existence of a complete $r$-graph with $(r+1)$ vertices is equal to the golden ratio $\tau=0.618\ldots$  for $r=2$, while it is $\frac{r}{r+1}$for $r\geq3$.
\end{abstract}

\section{Introduction}
One of the classical problems in extremal graph theory is that of finding the maximum density of a graph $G$ which does not contain some fixed graph $H$.  This density is known as the Tur\'an density for $H$ and is defined as $\pi(H)=\lim_{n\rightarrow \infty}\frac{ex(n,H)}{{n \choose 2}}$, where  the Tur\'an number $ex(n,H)$ is the maximum number of edges in a graph on $n$ vertices which does not have $H$ as a subgraph. Mantel  proved that $\pi(K_3)=\frac{1}{2}$ and later Tur\'an gave a complete answer for $H=K_t$ \cite{Tu}.  

Motivated by a question of Erd\H os regarding the maximum density of a triangle free subgraph of a graph \cite{BSTT} investigated a modified version of the Tur\'an density,  namely that of finding the maximum density of a 3-partite graph which does not contain a $K_3$. The problem was given a complete solution in terms of the three bipartite graphs induced by each pair of vertex classes of the 3-partition. Later \cite{BJT} investigated the number of $K_3$s as a function of these densities, and  a sharp result was given for large enough densities. Enumeration of triangles in general graphs has a long history and was finally solved by Razborov \cite{R2}. 

Our aim in this paper is to investigate generalisations of these questions to uniform hypergraphs. In particular we will determine the maximal density of an $(r+1)$-partite $r$-uniform hypergraph which does not contain the complete hypergraph $K_{r+1}^r$.  We will also present a sharp bound on the number of copies of $K_{r+1}^r$.   These results demonstrate a qualitative difference between graphs and hypergraphs with $r\geq 3$, where interestingly enough the extension to the hypergraph case is less complex than the graph case.  The existence condition found in \cite{BSTT} for graphs is non-linear in terms of the involved densities, as are the counting results from  \cite{BJT}, but as our results show the corresponding condition, and counting function, for $r\geq 3$ are given by  simple linear functions.

For hypergraphs far less is known in the non-partite case than for graphs. Tur\'an conjectured that $\pi(K_4^{3}) = \frac{5}{9}$, and gave a matching construction for the lower bound. Using flag-algebra  Razborov \cite{R} has proven that $\pi(K_4^{3}) \leq 0.56167$. For $r=4$ Giraud \cite{Gi} gave a construction which shows that $\pi(K_5^4)\geq \frac{11}{16}$, and Sidorenko \cite{Sid} conjectured that this is in fact the correct value. The best current upper bound $\pi(K_5^4)\leq \frac{1753}{2380}$ was given by the first author in \cite{Ma}.  For $r\geq 5$ much less is known.  De Caen \cite{dC} proved that $\pi(K_{r+1}^r)\leq (1-\frac{1}{r})$, and this was later sharpened somewhat for odd $r$ in \cite{CL} and even $r$ in \cite{LZ}.  As a corollary to one of our results we will get a short proof of de Caen's bound. 

Mubayi and Talbot \cite{MT} investigated the global density of a 4-partite 3-graph $G$ such that $K_4^{3}\not\subset G$. and proved the sharp result  $|E(G)|\leq\frac{8}{17} {|V(G)| \choose 3}$.  

One of the few counting results for hypergraphs is by  Mubayi \cite{DM} who  used the hypergraph removal lemma \cite{Go07,NRS,RS06,T06} to prove lower bounds on the number of copies of $H$ in a hypergraph with $ex(n,H)+q$ edges. These bounds apply to certain forbidden hypergraphs $H$ with the property that there is a unique $H$-free hypergraph on $ex(n,H)$ edges, and the bounds are  of the form $q c(n,H)$,  where $c(n,H)$ is the minimum number of copies of $H$ in any $n$-vertex $r$-graph on $ex(n,H)+1$ edges.

\subsection{Definitions}
For $r \geq 2$ we will refer to an $r$-uniform hypergraph as an $r$-graph. An $r$-graph $G$ is  $t$-partite  if its vertex has a partition in vertex classes  $V_1,\ldots,V_t$, such that each edge has at most one vertex in any class $V_i$.

\begin{definition}
	Given a $t$-partite $r$-graph $G$ with vertex classes $V_1,\ldots,V_t$ and a set of indices $I$ we let $P_I$ denote the $(t-|I|)$-partite $r$-graph induced 
	by the classes with indices not in $I$.  If $I=\{i\}$ then we abbreviate this as $P_i$ 
\end{definition}

As in \cite{BSTT} and \cite{BJT} we are going to work with vertex weighted $r$-graphs. We assume that $G$ has a weight-function $w$ defined on its vertices.  The weight of a set $S$ of vertices is $w(S)=\sum_{v\in S} w(v)$.

We define the weight of an edge $e=\{v_1,\ldots, v_r \}$ to be the product $\prod_{i=1}^r w(v_i)$, and the weight of a set of edges as the sum of their weights.

An unweighted $r$-graph will here be seen as the same graph with a weight function $w(v)=1$ for all vertices $v$.

\begin{definition}
	Given an $(r+1)$-partite $r$-graph $G$  we let $\rho(i)$ be the density of the $r$-partite graph induced by  $P_i$, which is 
	$\frac{w(E(P_i))}{\prod_{j\neq i}w(V_j)}$, and $\overline{\rho}$ is the vector of these densities.
\end{definition}

\section{Threshold and the minimum number of $K_{r+1}^r$ in an $(r+1)$-partite $r$-graph}

Given an $(r+1)$-partite $r$-graph we want to count the number of $K_{r+1}^r$s it contains. It will be convenient to do this in terms of the density of $K_{r+1}^r$s
\begin{definition}
	Given an $(r+1)$-partite $r$-graph  $G$ we define the density of $K_{r+1}^r$s in $G$ to be
	$$C(G)=\left (\sum_{S=K_{r+1}^r\subset G} \prod_{v \in S} w(s)\right) / \prod_{i=1}^{r+1}w(V_i)$$
	
	We define the minimum density of $K_{r+1}^r$s as 
	$$  C(\overline{\rho}) = \inf_{G} C(G), $$
	where the infimum is taken over all $(r+1)$-partite $r$-graphs with density $\overline{\rho}$.
\end{definition}

We will need the following polynomial from \cite{BJT}.
\begin{definition}
	$\Delta(a,b,c)=a^2+b^2+c^2-2ab-2ac-2bc+4abc$
\end{definition}
The proof of the following lemma is a routine, but somewhat lengthy,  calculus exercise.
\begin{lemma}\label{pos}
	If $a,b,c$ are real numbers in the interval $[0,1]$ and $a+b+c\geq \frac{9}{4}$ then $\Delta(a,b,c)\geq 0$, and  $a b + c >1$, $a c + b >1$, and $b c+a >1$.
\end{lemma}
In particular if each if each of $a, b$ and $c$ is at least $\frac{3}{4}$ then $\Delta(a,b,c)\geq 0$.

The following theorem was proven in \cite{BJT}.
\begin{theorem}\label{tridens}
	The minimum density of triangles  in a tripartite graph with edge densities given by $(a,b,c)$ such that $\Delta(a,b,c)\geq 0$,  $a b + c >1$, $a c + b >1$, and $b c+a >1$,  is given by $a+b+c-2$.
\end{theorem}

We will fist extend the lower bound given by this theorem to larger values of $r$.
\begin{lemma}\label{ub}
	If $\overline{\rho}$ is a  vector with rational numbers between 0 and 1, such that \\
	$$\sum_{i=1}^{r+1} \rho(i) - r \geq 0$$ \\
	then there exists a weighted $(r+1)$-partite $r$-graph $G$ with rational weights such that $\overline{\rho}$ is the density vector of $G$ and  $C(G)=\sum_{i=1}^{r+1} \rho(i)-r$.
\end{lemma}
\begin{proof}
	We will prove the statement by induction on $r$.  For $r=2$ the statement is true by Theorem 2.4  of \cite{BJT}, which is also Theorem 2.4 in the present paper.  By that theorem there is an optimal graph 
	which is a weighted 3-partite graph on six vertices whose 3-partite complement is a perfect matching. Hence we only need to prove the statement for $r\geq 3$.
	 
	 Without loss of generality we assume that $\rho(i)\geq\rho(i+1)$ for $1\leq i \leq r$ and hence, by averaging, that  
	 $\rho(1)+\rho(2)+\rho(3)\geq \frac{3r}{r+1}\geq\frac{9}{4}$. Thus, by Lemma \ref{pos}, we know that  the conditions of Theorem \ref{tridens} are satisfied by these three values.
	 
	 We now assume that the statement is true for $r-1$. In order to build an $r$-graph $H_2$ which satisfies the statement in the Lemma we take an $r$-partite 
	 $(r-1)$-graph $H_1$, with the first $r$ parts of $\overline{\rho}$ as its density vector, which satisfies the statement for $r-1$, and build an $(r+1)$-partite 
	 $r$-graph $H_2$ by adding one new class $V_{r+1}$ which contains a single vertex $v$, with weight 1.  
	 
	 For every edge $e$ in $H_1$ we let $e \cup \{v\}$ be an edge 	 of $H_2$. This means that for $i\neq r+1$  the density contributed to $\rho(i)$ by these edges will 
	 be the same as in $H_1$.  Hence $H_2$ has the desired density $\rho(i)$ for these classes, and the edges added so far do not give rise to a $K_{r+1}^r$.

	We will now add edges among the first $r$ classes.  One of these edges will be part of a weighted $K_{r+1}$ in $H_2$  if the corresponding $r$-tuple in $H_1$ is a $K_r^{(r-1)}$ in $H_1$, 
	and by induction we know the  density of such $K_{r}^{(r-1)}$.
	
	As pointed out in \cite{BJT}  we may assume that $H_1$ is in fact a simple unweighted hypergraph, since a hypergraph with rational weights can be modified by a suitable 
	blow-up into an equivalent unweighted hypergraph on a larger number of vertices.
	
	In  $H_1$  the density of partite $r$-tuples which do not span a $K_{r}^{(r-1)}$  is $1- \left( \sum_{i=1}^{r} \rho(i)-(r-1)\right)$, and all these tuples can be added as edges 
	without creating a $K_{r+1}^r$ in $H_2$.    Each additional edge added after these form a unique $K_{r+1^r}$ together with the vertex in $V_{r+1}$, so if we add enough 
	edges to reach the desired density $\rho(r+1)$ we will have 
	 $$  	  \rho(r+1) - \left(1- \left( \sum_{i=1}^{r} \rho(i)-(r-1) \right ) \right )   	 =  \sum_{i=1}^{r+1} \rho(i)-r $$
	 density of $K_{r+1}^r$ in $H_2$.	 	
\end{proof}

Our next step will to prove a lower bound for the density of $K_r$s.
\begin{theorem}\label{bt}
	Let $G$ be an $(r+1)$-partite $r$-graph, then the density of $K_{r+1}^r$s satisfies the following inequalities
	\begin{enumerate}
		\item	 $C(\overline{\rho})\geq \left(\sum_{i=1}^{r+1} \rho(i)\right )-r$.
		\item If $\overline{\rho}$ satisfies the conditions of Lemma \ref{ub}  then $C(\overline{\rho})=\left(\sum_{i=1}^{r+1} \rho(i)\right )-r.$
	\end{enumerate}
\end{theorem}
\begin{proof}
	\begin{enumerate}
		\item Let $\mathbf{1}(e)$, where $e$ is a set of size $r$,  be 1 if $e$ is an edge of $G$ and 0 otherwise, let $\mathbf{1}(H)$ be 1 if $H$ is the 
		vertex set of a $K_{r+1}$ contained in $G$, and 0 otherwise.
	
		Now 
		$$ \sum_{H} w(H)( \sum _{i \in H }\mathbf{1}(H\setminus i)   ) \leq   \sum_{H} w(H)( r + \mathbf{1}(H)  )  $$
	
		But the left hand side is the sum of the densities in $G$ and the right hand side is $r$ plus the density of $K_{r+1}^r$s so
	
		$$\sum_i \rho(i) \leq r + C(\overline{\rho})$$

		and so $C(\overline{\rho}) \geq \sum_i \rho(i) - r$, and we have proven part 1.
	
		\item Lemma \ref{ub} shows that the bound in Theorem \ref{bt}  is sharp, and that for rational densities satisfying the inequality in Lemma 2.4 equality is achieved for some finite $r$-graph.
	\end{enumerate}	
\end{proof}

\begin{corollary}
	If all densities $\rho(i)> \frac{r}{r+1}$ in $G$ then $K_{r+1}^r \subset G$ 
	
	If $r\geq 3$  and  $\rho(i)=\rho(j)$, for all $i,j$ and $\rho(i)\leq\frac{r}{r+1}$ then there exists $G$ with these densities such that $K_{r+1}^r \not\subset G$ 
\end{corollary}

We note that the graph which achieves the minimum number of copies of $K_{r+1}^r$s is not unique, i.e.the $(r+1)$-partite Tur\'an-problem for $K_{r+1}^r$ is not stable.  Byt the existing results for $r=2$  we know that 
there is not a unique graph which minimizes the number of triangles and Lemma \ref{ub} gives distinct extensions to higher values of $r$. 

\section{Balanced Codegrees}
Our second result concerns degrees rather than densities, and for $r$-graph we have found it natural to consider the degrees of the $(r-1)$-tuples of vertices in $G$.

\begin{definition}
	Given a multipartite $r$-graph $G$, with $V(G)$ partitioned into at least $r$ classes, we say that a $t$-tuple $g$ of vertices from $G$ is partite if it has at most one vertex in each class of $G$.
\end{definition}

\begin{definition}
	Given a partite $t$-tuple $g$ we say that an $(r-t)$-tuple $h$ is completing $g$ if $(g \cup h)\in E(G)$. We call the set of completing $(r-t)$-tuples for $g$ the neighbourhood $N(g)$ of $g$.  The neighbourhood of $g$ in a set $I$ of classes is the set of completing $(r-t)$-tuples in $I$, and is denoted $N(I,g)$	
	
	The degree  of $g$ is $d(g)=|N(g)|$ and the degree in an $(r-t)$-tuple $I$ of classes is $d(I,g)=|N(I,g)|$.		
\end{definition}

By the minimum codegree of an (multipartite) $r$-graph $G$ we refer to the minimum degree of all (partite) $(r-1)$-tuples of vertices in $V(G)$

\begin{definition}
	We say that a partite $t$-tuple $g$ of vertices in $G$ has strictly balanced degree if it has the same number of 
	completing $(r-t)$-tuples in each of the $(r-t)$-tuples of  classes which $g$ does not intersect.	
\end{definition}

\begin{theorem}\label{degth}
	If the partite $(r-1)$-tuples of $G$ have strictly balanced degrees and  $\max_j \sum_{i \neq j} \rho(i) > (r-1)$  then  $K_{r+1}^r\subset G$.
\end{theorem}
Note that the condition on the partite $(r-1)$-tuples  means that such a tuple splits its neighbourhood equally between each of the two parts which the tuple does not intersect, but the sizes of those neighbourhoods may differ between different tuples.

\begin{proof}
	We will first look at the case $j=1$.  Let $e$ be an edge of $P_1$ and let $g_2,\ldots, g_{r+1}$ be the $(r-1)$-tuples which are subsets of $e$.  
	If $G$ does not contain $K_{r+1}^r$ we must have that 
	$$ S(e)= \sum_{i=2}^{r+1} d(V_1,g_i) \leq (r-1), $$
	since otherwise there would be a common vertex  in the neighbourhoods of the $(r-1)$-tuples, and we would have a $K_{r+1}^r$ .

	If we sum over all edges in $P_1$ we find that
	$$\sum_{e\in E(P_1)} S(e) \leq (r-1)\rho(1) \prod_{k\neq 1}|V_k| $$
	or equivalently
	$$\sum_{i=2}^{r+1}  \sum_{g\in P_{\{1,i\} }} d(V_1,g)d(V_i,g) \leq (r-1)\rho(1) \prod_{k\neq 1}|V_k|   $$
	 where the inner sum is over all partite $(r-1)$-tuples $g$.

	Using the fact that each $g$ has strictly balanced degree we rewrite this as
	$$\sum_{i=2}^{r+1}  \sum_{g\in P_{\{1,i\}} } d(V_1,g)^2  \leq (r-1)\rho(1) \prod_{k\neq 1}|V_k|   $$

	Using the Cauchy-Schwarz inequality we get
	
	$$\sum_{i=2}^{r+1}  \frac{(\sum_{g\in P_{\{1,i\}} } d(V_1,g))^2 }{ \prod_{k\neq 1,i}|V_k|    }        \leq (r-1)\rho(1) \prod_{k\neq 1}|V_k|   $$

	By definition 
	$$\sum_{g\in P_{\{1,i\}} } d(V_1,g) = \rho(i) \prod_{k\neq i }|V_k| . $$

	Substituting this we find that 
	$$\sum_{i=2}^{r+1}  \rho(i)^2\prod_{k\neq i}|V_k|   \leq (r-1)\rho(1) \prod_{k\neq 1}|V_k|   $$

	But $\rho( i )\prod_{k\neq i }|V_k| $ is independent of $i$ so we can divide both sides by $\rho(1) \prod_{k\neq 1}|V_k|  $ to get
	
	$$\sum_{i=2}^{r+1}  \rho(i)  \leq (r-1)    $$

	The results follows for other values of $j$ in the same way.
	
\end{proof}
As shown in  \cite{BSTT} this result is sharp for $r=2$, but we do not have a matching lower bound for larger $r$.

Relating to the case where all densities are equal we get the following
\begin{corollary}\label{co1}
	If all partite $(r-1)$-tuples of $G$ have strictly  balanced degree and  $\rho(i)> 1-\frac{1}{r}$ for all $i$ then $K_{r+1}^r\subset G$.
\end{corollary}
This can be compared to the minimum codegree which forces a $K_{r+1}^r$ in  the non-partite case. In \cite{LM} it was proven that there are $r$-graph with minimum $(r-1)$-degree  $\frac{1}{2}(n-2)$ which do not contain a $K_{r+1}^r$, and it was conjectured that this is an optimal bound. This in turn implies that the global density is at least $\frac{1}{2}$. 

Another simple corollary of this result is de Caen's upper bound on the Tur\'an-density of $K_{r+1}^r$.
\begin{corollary}[de Caen \cite{dC}]\label{dc}
	If $G$ is an $r$-graph and  $|E(G)|> (1-\frac{1}{r})\frac{n^r}{r!}$ then $K_{r+1}^r\subset G$.
 \end{corollary}
\begin{proof}

	Given a labelled $r$-graph $G$ with vertex set $V$ we will form a new $(r+1)$-partite $r$-graph $H$. The vertex set of $H$ is the disjoint union of $r+1$ copies $V_1,\ldots, V_{r+1}$ of $V$.
	
	Given an edge $e=\{v_1,v_2,...,v_r\}$ of $G$,  a choice of $r$  of the classes $V_i$, and a permutation $\pi\in S_r$,  we let    $e_{\pi}=\{v_{\pi(1)},v_{\pi(2)}, \ldots, v_{\pi(r)}  )$, where the 
	$j$th vertex in the tuple is the vertex in the $j$th of the $r$ vertex classes,  be an edge in the $r$-partite subgraph of $H$ induced by those classes.   Note that  in this way each edge $e$ of 
	$G$  gives 	rise to $r!$ edges in each of the $r$-partite subgraphs of $H$. (So if we look at ordinary graphs and let $G$ be a single edge  (1,2)  then $H$ would become a 6-cycle which 
	winds twice through the three parts.)
		
	The $r$-graph $H$ has strictly balanced codegrees, since the number of neighbours in $V_i$, of a partite $(r-1)$-tuple $t$ in $H$, only depends on whether $t$ intersects $V_i$ or not, and the 
	number of neighbours of $t$ in $G$.
	
	By the assumptions each $\rho(i)$  in $H$ is strictly greater than $(1-\frac{1}{r})$, so $H$ contains a copy of $K_{r+1}^r$, and the corresponding vertices form a $K_{r+1}^r$ in $G$.
\end{proof}

In fact the proof of Theorem \ref{degth} can easily be modified to give a bound for the degree required to give a copy of $(K_{r+1}^r-k)$, the  $r$-graph obtained by deleting any $k$ edges from $K_{r+1}^r$, we just need to modify the first bound for $S(e)$ to be less than $r-k-1$. This gives 
\begin{theorem}\label{degth2}
	If the partite $(r-1)$-tuples of $G$ have strictly balanced degrees, $k\leq r-1$, and  $\max_j \sum_{i \neq j} \rho(i) > r-k-1$  then   $(K_{r+1}^r-k)\subset G$.
\end{theorem}
For $k\geq r-1$ the graph $(K_{r+1}^r-k)$  is in fact $r$-partite, and the strictly balanced degree condition means that if  $G$ has at least one edge then any edge will be part of a $(K_{r+1}^r-(r-1))$, which is simply two edges overlapping on an $(r-1)$-tuple.

Using this theorem as in the proofs of Corollaries \ref{co1} and \ref{dc} we get
\begin{corollary}
	If all partite $(r-1)$-tuples of $G$ have balanced degree, $k\leq (r-1)$, and  $\rho(i)> 1-\frac{k+1}{r}$ for all $i$ then  $(K_{r+1}^r-k)\subset G$.
\end{corollary}
\begin{corollary}
	If $G$ is an $r$-graph, $k\leq(r-1)$, and  $|E(G)|> (1-\frac{k+1}{r}) \frac{n^r}{r!}$ then $(K_{r+1}^r-k)\subset G$.
 \end{corollary}
The case $r=3$ of the latter result was proven by de Caen in \cite{dC}.

\section{Open problems}
Following \cite{BSTT} and the results in this paper concerning the codegree it is natural to ask what happens for general degrees in $r$-graphs with $r>2$. 
The first open case is vertex degrees for $3$-graphs
\begin{question}
	Let $G$ be a 4-partite  $3$-graph with balanced vertex degrees. Which densities force a $K_{4}^3$ in $G$?
\end{question}

For $r>3$ the corresponding question is open for balanced $l$-degrees for all $l<r-1$. 
\begin{question}
	Let $G$ be an $r$-partite  $r$-graph with $4\leq r$. Which densities forces a $K_{r+1}$ if the partite $l$-tuples have balanced degrees, where $1\leq l<r-1$.
\end{question}

Let $G$ be a $t$-partite $r$-graph with all densities at least $\alpha$. Given an $r$-graph $H$, we can ask for the minimum density $\alpha$ which forces a copy 
of $H$ in $G$.  We call this density $d_t(H)$.
\begin{question}
	Does there exist a finite $t_0$ such that   $d_{t}(H)=d_{t_0}(H)$ for  $t\geq t_0$?
\end{question}
For $r=2$ and $H=K_t$ the answer is yes by \cite{pf}.

\section*{Acknowledgement}
This research was done while the authors were attending the research semester Graphs, hypergraphs and computing at  
Institut Mittag-Leffler (Djursholm, Sweden).  The first author was supported by  The Swedish Research Council grant 2014--4897.
The second author was supported by ERC Advanced Grant GRACOL.

\providecommand{\bysame}{\leavevmode\hbox to3em{\hrulefill}\thinspace}
\providecommand{\MR}{\relax\ifhmode\unskip\space\fi MR }
\providecommand{\MRhref}[2]{%
  \href{http://www.ams.org/mathscinet-getitem?mr=#1}{#2}
}
\providecommand{\href}[2]{#2}

\end{document}